\documentclass[11pt,reqno]{amsart}
\usepackage{amsmath,amsthm,amssymb}

\usepackage{graphics}
\usepackage{hyperref}
\usepackage[usenames, dvipsnames]{xcolor}

\usepackage[square,sort,comma,numbers]{natbib}

\definecolor{darkblue}{rgb}{0.0,0.0,0.3}
\hypersetup{colorlinks,breaklinks,
  linkcolor=darkblue,urlcolor=darkblue,
anchorcolor=darkblue,citecolor=darkblue}

\theoremstyle{plain}
\newtheorem{theorem}{Theorem}[section]
\newtheorem*{theorem*}{Theorem}
\newtheorem{lemma}[theorem]{Lemma}
\newtheorem{proposition}[theorem]{Proposition}
\newtheorem*{proposition*}{Proposition}
\newtheorem{corollary}[theorem]{Corollary}
\newtheorem*{corollary*}{Corollary}

\theoremstyle{definition}
\newtheorem{remark}[theorem]{Remark}

\numberwithin{equation}{section}

\renewcommand{\Im}{\operatorname{Im}}
\renewcommand{\Re}{\operatorname{Re}}
\DeclareMathOperator{\SL}{SL}

\title{Short-Interval Averages of Sums of Fourier Coefficients of Cusp Forms}
\author[Hulse]{Thomas A. Hulse}
\thanks{Research of the first author was supported by a Coleman Postdoctoral Fellowship at Queen's University.}

\author[Kuan]{Chan Ieong Kuan}
\author[Lowry-Duda]{David Lowry-Duda}
\thanks{The third author is supported by the National Science Foundation Graduate Research Fellowship Program under Grant No. DGE 0228243.}

\author[Walker]{Alexander Walker}
\begin{document}

\maketitle

\begin{abstract}
  Let $f$ be a weight $k$ holomorphic cusp form of level one, and let $S_f(n)$ denote the sum of the first $n$ Fourier coefficients of $f$.
  In analogy with Dirichlet's divisor problem, it is conjectured that $S_f(X) \ll X^{\frac{k-1}{2} + \frac{1}{4} + \epsilon}$.
  Understanding and bounding $S_f(X)$ has been a very active area of research.
  The current best bound for individual $S_f(X)$ is $S_f(X) \ll X^{\frac{k-1}{2} + \frac{1}{3}} (\log X)^{-0.1185}$ from Wu~\cite{wu2009power}.

  Chandrasekharan and Narasimhan~\cite{chandrasekharan1962functional} showed that the Classical Conjecture for $S_f(X)$ holds on average over intervals of length $X$.
  Jutila~\cite{jutila1987lectures} improved this result to show that the Classical Conjecture for $S_f(X)$ holds on average over short intervals of length $X^{\frac{3}{4} + \epsilon}$.
  Building on the results and analytic information about $\sum \lvert S_f(n) \rvert^2 n^{-(s + k - 1)}$ from our recent work~\cite{hkldw}, we further improve these results to show that the Classical Conjecture for $S_f(X)$ holds on average over short intervals of length $X^{\frac{2}{3}}(\log X)^{\frac{1}{6}}$.
\end{abstract}

%%% Begin actual document %%%

\section{Introduction and Statement of Results}

Let $f$ be a holomorphic cusp form of integer weight $k \geq 2$ on $\SL_2(\mathbb{Z})$, %a simultaneous eigenform of all the Hecke operators,
and let the Fourier expansion of $f$ at the cusp at infinity be given by
\[
  f(z) = \sum_{n \geq 1} a(n)e(nz),
\]
where $e(z) := e^{2\pi iz}$. We adopt the convention that $a(n) = 0$ for $n \leq 0$. Let $S_f(n)$ denote the $n$th partial sum,
\begin{equation*}
  S_f(n) := \sum_{m \leq n} a(m).
\end{equation*}

Folowing Deligne's celebrated bound~\cite{Deligne}, we have $a(m) \ll m^{\frac{k-1}{2}+\epsilon}$, which gives the trivial bound $S_f(n) \ll n^{\frac{k-1}{2}+1+\epsilon}$.
One might assume that the signs of each $a(m)$ are random, which would produce a ``square-root cancellation estimate'' $S_f(n) \ll n^{\frac{k-1}{2} + \frac{1}{2} + \epsilon}$.
However, significantly more cancellation is known.
Deligne~\cite{Deligne} showed that
\begin{equation}\label{eq:one_third_with_epsilon}
  S_f(n) \ll n^{\frac{k-1}{2} + \frac{1}{3} + \epsilon}.
\end{equation}
Jutila~\cite{jutila1987lectures} and Hafner and Ivi\'{c}~\cite{HafnerIvic89} removed the $\epsilon$ in the above bound and Rankin~\cite{rankin1990sums} showed how to get logarithmic savings. More recently, Wu~\cite{wu2009power} showed that
\begin{equation*}
  S_f(n) \ll n^{\frac{k-1}{2} + \frac{1}{3}} (\log n)^{-0.1185}
\end{equation*}
and used this result to improve lower bounds on numbers of Hecke eigenvalues of the same signs.

In the other direction, Hafner and Ivi\'{c}~\cite{HafnerIvic89} showed that, for some positive constant $D$,
\begin{equation*}
  S_f(n) = \Omega_{\pm} \left( n^{\frac{k-1}{2} + \frac{1}{4}} \left(\frac{D (\log \log x)^{1/4}}{(\log \log \log x)^{3/4}}\right)\right).
\end{equation*}
After some normalization, this result resembles the conjectured $\frac{1}{4}$-exponent in the Gauss circle and Dirichlet divisor problems, and similarly it is natural to investigate the best exponent, $\alpha$, such that $S_f(n) \ll n^{\frac{k-1}{2}+\alpha+\epsilon}$. As with the circle and divisor problems, it is speculated that
\begin{equation}\label{eq:classical_conjecture}
  S_f(n) \ll n^{\frac{k-1}{2} + \frac{1}{4} + \epsilon}
\end{equation}
and we refer to this hypothesis as the ``Classical Conjecture,'' after Hafner and Ivi\'{c}.

Chandrasekharan and Narasimhan~\cite{chandrasekharan1962functional, chandrasekharan1964mean} proved that the Classical Conjecture is true \emph{on average} in long intervals by showing that for $X\gg 1$,
\begin{equation}\label{eq:CN_classical_long_intervals}
  \frac{1}{X} \sum_{n \leq X} \vert S_f(n) \vert^2 = C X^{k - 1 + \frac{1}{2}} +O\left(X^{k-1+\epsilon}\right),
\end{equation}
with an explicit constant $C$, which gives~\eqref{eq:classical_conjecture} on average via the Cauchy-Schwarz inequality.

Jutila~\cite{jutila1987lectures} improved~\eqref{eq:CN_classical_long_intervals} by showing that the Classical Conjecture is true on average in intervals of length $X^{\frac{3}{4} + \epsilon}$ around $X$.
Very recently, Ernvall-Hyt\"{o}nen~\cite{ernvall2015mean} extended Jutila's work and investigated average orders of short sums of Fourier coefficients.

Short-interval average results can be used to derive bounds for individual $S_f(n)$.
In particular, the Classical Conjecture would follow from a proof of the Classical Conjecture \emph{on average} in intervals of length $X^{\frac{1}{4}+\epsilon}$ around $X$.

Here, we show that the Classical Conjecture holds on average in intervals of length less than $X^{\frac{3}{4}+\epsilon}$. In particular, we prove the following theorem (we actually prove slightly more, see Remark~\ref{rmk:main_theorem}), which shows that the Classical Conjecture is true on average within intervals around $X$ of length $X^{\frac{2}{3}}(\log X)^{\frac{1}{6}}$. 

\begin{theorem}\label{thm:classical_conjecture_in_twothird_intervals}
  Let $f$ be a full integer weight $k \geq 2$
	%Hecke
	cusp form on $\SL_2(\mathbb{Z})$. Then
  \begin{equation*}
    \frac{1}{X^{\frac{2}{3}} (\log X)^{\frac{1}{6}}} \sum_{\lvert n - X \rvert < X^{\frac{2}{3}} (\log X)^{\frac{1}{6}}} \lvert S_f(n) \rvert^2 \ll X^{k-1 + \frac{1}{2}}.
  \end{equation*}
\end{theorem}

In~\cite{hkldw}, the authors investigated the behaviour of the second moment of the partial sums, $S_f(n)$, of the Fourier coefficients.
The authors provided meromorphic continuation for the Dirichlet series,
\begin{equation*}
  D(s, S_f \times S_f) := \sum_{n \geq 1} \frac{\lvert S_f(n) \rvert^2}{n^{s + k - 1}},
\end{equation*}
and used this continuation to prove smoothed forms and generalizations of~\eqref{eq:CN_classical_long_intervals}. In this paper, we build upon the meromorphic continuation of $D(s, S_f \times S_f)$ from~\cite{hkldw} and apply a concentrating integral transform. Indeed, this result demonstrates how an understanding of the analytic properties of $D(s, S_f \times S_f)$ can produce new avenues for investigating the behaviour of $S_f(n)$.

\section{Decomposition and Outline of Proof}

In this section, we decompose $D(s, S_f \times S_f)$ into pieces we can understand and introduce the concentrating integral transform that leads to our main theorem.
The majority of this paper is dedicated towards bounding contributions from each part of this decomposition.
We outline the main argument of the paper, leaving proofs to later sections.

From Section~3 of~\cite{hkldw}, we have the following decomposition of $D(s, S_f \times S_f)$.
\begin{proposition}\label{prop:D_decomposition}
  Let $\displaystyle f(z) = \sum_{n\geq 1} a(n)e(nz)$ be a weight $k$ cusp form. Then
  \begin{align*}
    &D(s, S_f \times S_f) := \sum_{n \geq 1} \frac{\lvert S_f(n) \rvert^2}{n^{s + k - 1}} \nonumber \\
    &\quad= W(s; f,f) + \frac{W(s-1; f,f)}{s + k - 2}\\
    &\qquad + \frac{1}{2\pi i} \int_{(\sigma_z)} W(s-z; f,f) \zeta(z) \frac{\Gamma(z) \Gamma(s + k - 1 - z)} {\Gamma(s + k - 1)}dz, \\
  \end{align*}
  where $\Re \sigma_z \in (0,1)$, $\Re s$ is sufficiently positive, and
  \begin{align*}
    W(s; f,f) :\!\! &= \frac{L(s, f\times f)}{\zeta(2s)} + Z(s, 0, f\times f),
  \end{align*}
  in which $Z(s, w, f\times f)$ is the symmetrized shifted convolution sum
  \begin{equation*}
    Z(s, w, f\times f) := \sum_{n, h \geq 1} \frac{a(n) \overline{a(n-h)}+a(n-h)\overline{a(n)}}{n^{s + k - 1}h^w}.
  \end{equation*}
\end{proposition}

Here and throughout, we maintain the notation from~\cite{hkldw}, which largely reflects the notation in~\cite{HoffsteinHulse13}.
The analytic properties of $D(s, S_f \times S_f)$ are given by the analytic properties of the Rankin-Selberg convolution $L(s, f\times f)$ and the shifted convolution sum $Z(s, w, f\times f)$. We understand $L(s, f\times f)$ through the well-known relation
\begin{equation*}
  L(s, f\times f) = \frac{(4\pi)^{s + k - 1} \zeta(2s)}{\Gamma(s + k - 1)} \left \langle \lvert f \rvert^2 \Im(\cdot)^k, E(\cdot, \overline{s}) \right \rangle,
\end{equation*}
where $E(z,s)$ is the real analytic Eisenstein series
\begin{equation*}
  E(z,s) := \sum_{\gamma \in \Gamma_\infty \backslash \SL_2(\mathbb{Z})} \Im(\gamma z)^s.
\end{equation*}
We understand $Z(s, w, f\times f)$ through meromorphic continuation of its spectral expansion.
Let $\{\mu_j(z) : j \geq 0\}$ be an orthonormal basis of weight zero Maass eigenforms for the residual and cuspidal spaces of $\SL_2(\mathbb{Z}) \backslash \mathbb{H}$. This basis consists of the constant function $\mu_0(z)$ and infinitely many Maass cusp forms $\mu_j(z)$, with $j \geq 1$. To these cusp forms we associate eigenvalues $\tfrac{1}{4} +t_j^2$ with respect to the hyperbolic Laplacian, and Fourier expansions
\[
  \mu_j(z)=\sum_{n \neq 0} \rho_j(n)y^{\frac{1}{2}}K_{it_j}(2\pi \vert n \vert y)e^{2\pi i n x}.
\]
We may assume as well that each $\mu_j$ is a simultaneous eigenfunction of the Hecke operators, including the $T_{-1}$ operator which has action $T_{-1} \mu_j(x+iy) = \mu_j(-x + iy)$, as described in~\cite[Theorem~3.12.6]{Goldfeld2006automorphic}.

Following \cite{hkldw}, let
\begin{equation*}
  \mathcal{V}_{f,f}(z) = \Im(z)^k(\vert f(z) \vert^2 + T_{-1}(\vert f(z) \vert^2)).
\end{equation*}
The following Proposition gives the spectral expansion of $Z(s,w,f\times f)$ in terms of the Eisenstein series and the Maass eigenforms defined above.

\begin{proposition}[From Section~4 of~\cite{hkldw}]\label{prop:Z_meromorphic_continuation}
In the region $\Re s > \frac{1}{2}$ and $Re (s + w) > \frac{3}{2}$, the shifted convolution sum $Z(s,w, f\times f)$ can be expressed as
  \begin{align}
    Z(s,& w, f\times f) =\nonumber\\
    &\frac{(4\pi)^k}{2} \sum_j \rho_j(1) G(s, it_j) L(s+w - \tfrac{1}{2}, \mu_j) \langle \mathcal{V}_{f,f}, \mu_j \rangle \label{line:definition_discrete_spectrum}\\
    &+ \frac{(4\pi)^k}{4 \pi i} \int_{(0)} G(s,z) \mathcal{Z}(s,w,z) \langle \mathcal{V}_{f,f}, \overline{E(\cdot, \tfrac{1}{2} - z)}\rangle dz. \label{line:definition_cont_spectrum}
  \end{align}
  We call~\eqref{line:definition_discrete_spectrum} the discrete spectrum component and~\eqref{line:definition_cont_spectrum} the continuous spectrum component. Here, $G(s,z)$ and $\mathcal{Z}(s,w,z)$ are the collected gamma and zeta factors of the discrete and continuous spectra,
  \begin{align*}
    G(s,z) &:= \frac{\Gamma(s-\frac{1}{2}+z)\Gamma(s-\frac{1}{2}-z)}{\Gamma(s)\Gamma(s+k-1)}, \\
    \mathcal{Z}(s,w,z) &:= \frac{\zeta(s+w-\frac{1}{2}+z)\zeta(s+w-\frac{1}{2}-z)}{\zeta^*(1+2z)},
  \end{align*}
in which $\zeta^*(2z):=\pi^{-z} \Gamma(z)\zeta(2z)$ is the completed zeta function. The convolution $Z(s, 0, f\times f)$ has a meromorphic continuation to the whole plane given by the separate meromorphic continuations of its discrete and continuous spectra.
  \end{proposition}

Specializing to $w = 0$, the discrete spectrum component has a meromorphic continuation to the plane given by the analytic continuations of $L(s, \mu_j)$.
  The continuous spectrum also has a meromorphic continuation, but it is significantly more involved due to interaction between poles in $s$ and the Mellin integral in $z$.
  The meromorphic continuation of the continuous spectrum component can be written as
  \begin{equation} \label{eq:meromorphic_continuation}
    \frac{(4\pi)^k}{4\pi i} \int_{(0)} G(s,z) \mathcal{Z}(s,0,z) \langle \mathcal{V}_{f,f}, \overline{E(\cdot, \tfrac{1}{2} - z)}\rangle dz + \sum_{0 \leq m < \frac{3}{2} - \Re s} \rho_{\frac{3}{2} - m}(s),
  \end{equation}
  where
  \begin{equation}\label{eq:zeta_residual}
    \rho_{\frac{3}{2}}(s) = \frac{(4\pi)^k \zeta(2s - 2) \Gamma(2s - 2)}{\Gamma(s) \Gamma(s + k - 1) \zeta^*(2s - 2)} \langle \mathcal{V}_{f,f}, \overline{E(\cdot, 2 - s)}\rangle
  \end{equation}
  denotes the residual term coming from the apparent pole of the zeta functions in $\mathcal{Z}(s,0,z)$, and for $m \geq 1$,
	\begin{equation}
  \begin{split}
      \rho_{\frac{3}{2} - m}(s) &= \frac{(-1)^{m-1} (4\pi)^k \zeta(1-m) \zeta(2s + m-2) \Gamma(2s + m-2)}{\Gamma(m) \Gamma(s) \Gamma(s + k - 1) \zeta^*(4 - 2s - 2m)} \times \\
      &\qquad \big\langle \mathcal{V}_{f,f}, \overline{E(\cdot, s + m-1)}\big\rangle
    \end{split}\label{eq:jth_residual}
  \end{equation}
  denotes residual terms coming from apparent poles of the Gamma functions in $G(s,z)$.
  Each residual term $\rho_{\frac{3}{2} - m}(s)$ appears only when $\Re s < \tfrac{3}{2} - m$.
  For example, the first residual term, $\rho_{\frac{3}{2}}(s)$, appears when $\Re s < \frac{3}{2}$, and the next residual term, $\rho_{\frac{1}{2}}(s)$, appears when $\Re s < \frac{1}{2}$. When $Re(s)=\frac{3}{2}-m$, the line of integration in \eqref{eq:meromorphic_continuation} is bent so that it does not pass through poles of the integrand. 

We use the Mellin transform,
\begin{equation}\label{eq:integral_transform}
  \frac{1}{2\pi i} \int_{(2)} \exp{\left( \frac{\pi s^2}{y^2}\right)} \frac{X^s}{y} ds = \frac{1}{2\pi} \exp\left( -\frac{y^2 \log^2 X} {4\pi} \right),
\end{equation}
for $X \gg 1$ and $y>0$, to concentrate sums around the coefficients of $D(s, S_f \times S_f)$ in an interval.
In Section~\ref{sec:integral_transform} we prove~\eqref{eq:integral_transform} and investigate a family of related Mellin transforms.
Note that
\begin{align}
  &\frac{1}{2\pi i} \int_{(\sigma_s)} D(s, S_f \times S_f) \exp \left( \frac{\pi s^2}{y^2} \right) \frac{X^s}{y} ds \label{line1:integral_transform} \\
  %&\quad= \frac{1}{2\pi i} \int_{(\sigma)} \sum_{n \geq 1} \frac{\lvert S_f(n) \rvert^2}{n^{s + k - 1}} \exp\left(\frac{\pi s^2}{y^2}\right) \frac{X^s}{y} ds \\
  &\quad= \frac{1}{2\pi} \sum_{n \geq 1} \frac{\lvert S_f(n) \rvert^2}{n^{k-1}} \exp\left( - \frac{y^2 \log^2 (X/n)}{4\pi} \right), \label{line2:integral_transform}
\end{align}
when $\Re \sigma_s$ is large enough such that $D(s, S_f \times S_f)$ converges absolutely.
Taking $\Re \sigma_s > 3$ suffices.
For $n \in [X - X/y, X + X/y]$, the exponential damping term in line~\eqref{line2:integral_transform} is almost constant.
But for $n$ with $\lvert n - X \rvert > X/y^{1-\epsilon}$, the damping term contributes exponential decay.
So the main contribution to the integral in~\eqref{line1:integral_transform} concentrates around those coefficients in an interval around $X$.
Further, by the positivity of the coefficients,
\begin{equation}\label{eq:goal_inequality}
  \sum_{\lvert n - X \rvert < X/y} \frac{\lvert S_f(n) \rvert^2}{n^{k-1}} \ll \frac{1}{2\pi i} \int_{(\sigma_s)} D(s, S_f \times S_f) \exp \left( \frac{\pi s^2}{y^2} \right) \frac{X^s}{y} ds.
\end{equation}

We split $D(s, S_f \times S_f)$ into the three components given by the decomposition in Proposition~\ref{prop:D_decomposition} and bound the integral transform on the right hand side of~\eqref{eq:goal_inequality} separately for each piece.
Our primary goal then becomes to bound each of the components below,
\begin{align}
  &\frac{1}{2\pi i} \int_{(\sigma_s)} W(s; f,f) \exp \left( \frac{\pi s^2}{y^2} \right) \frac{X^s}{y} ds \label{line1:decomposition} \\
  &+ \frac{1}{2\pi i} \int_{(\sigma_s)} \frac{W(s-1; f,f)}{s + k - 2} \exp \left( \frac{\pi s^2}{y^2} \right) \frac{X^s}{y} ds \label{line2:decomposition} \\
  \begin{split}
    &+ \frac{1}{(2\pi i)^2} \int_{(\sigma_s)}\int_{(\sigma_z)} W(s-z; f\times f)\zeta(z) \frac{\Gamma(z)\Gamma(s + k - 1 - z)}{\Gamma(s + k - 1)}  \label{line3:decomposition} \\
&\mkern 200mu \times\exp \left( \frac{\pi s^2}{y^2} \right) \frac{X^s}{y} dz ds. \end{split}
\end{align}
For the first term~\eqref{line1:decomposition}, coming from $W(s; f,f)$, it is very straightforward to produce the following upper bound.
\begin{lemma}\label{lem:Wsff_upper_bound}
  \begin{equation*}
    \int_{(\sigma_s)} W(s; f,f) \exp\left(\frac{\pi s^2}{y^2}\right) \frac{X^s}{y}ds \ll_\epsilon \frac{X^{\frac{4}{3} + \epsilon}}{y}.
  \end{equation*}
\end{lemma}
\begin{proof}
  See Section~\ref{sec:Wsff}.
\end{proof}

The third term, coming from~\eqref{line3:decomposition}, the $z$-integral of $W(s-z; f,f)$,
requires more attention.
We fix $\delta > 0$ and introduce the assumption that $y \ll X^{1-\delta}$, eventually taking $\delta$ near $\frac{2}{3}$.
Under this assumption, we are able to shift our lines of integration for both the $z$ and $s$ integrals in a way that minimizes our final estimates.

To be specific, with $y \ll X^{1-\delta}$, it becomes advantageous to shift $\sigma_s$ to be very negative and $\sigma_z$ to be very positive even though these shifts pass over several poles.
The number of residual contributions from these poles depends on the choice of $\delta$, and the resulting implicit constant might be very large in $\delta$.
Ultimately, we are able to show that the contribution from \eqref{line3:decomposition}
is no larger than that seen in Lemma~\ref{lem:Wsff_upper_bound}, up to implicit dependence on $\delta$.

\begin{lemma}\label{lem:MellinBarnesWsff_upper_bound}
  \begin{align*}
    \begin{split}
      \int_{(\sigma_s)}\int_{(\epsilon)} &W(s-z; f,f) \zeta(z) B(s+k-1-z,z) \exp\left(\frac{\pi s^2}{y^2}\right) \frac{X^s}{y} dzds \\
      &\ll_{\delta, \epsilon} \frac{X^{\frac{4}{3} + \epsilon}}{y}
    \end{split}
  \end{align*}
  where $B(s,z)$ is the Beta function
  \begin{equation*}
    B(s,z) = \frac{\Gamma(z) \Gamma(s)}{\Gamma(s+z)}.
  \end{equation*}
\end{lemma}

\begin{proof}
  See Section~\ref{sec:Wsff_Mellin_Barnes_Integral}.
\end{proof}

For the term~\eqref{line2:decomposition},
coming from $W(s-1; f,f)/(s+k-2)$, we proceed with extreme care as
this term dominates the previous two. It is necessary to analyze $W(s; f,f)$ in two smaller pieces.

In Section~\ref{ssec:Wsff_residue_continuous} we consider the contribution from $L(s, f\times f) \zeta(2s)^{-1}$ and the continuous spectrum of $Z(s, 0, f\times f)$, which we group together to take advantage of the remarkable cancellation of polar terms.
In our methodology, it becomes necessary to further assume that $y \ll X^{1 - \delta}$ with $\delta > \tfrac{1}{2}$.
With this extra assumption, it is advantageous to shift the line of $s$-integration far to the left.

Recall from~\eqref{eq:meromorphic_continuation} that the meromorphic continuation of $Z(s, 0, f\times f)$ gains additional residual terms as $\Re s$ is taken further negative.
These residual terms contribute the two dominant bounds in our short-interval estimate.
The first residual term $\rho_{\frac{3}{2}}(s)$ has a pole which contributes $O_\delta(X^{\frac{3}{2}}/y)$, while the rest of the residual terms contribute $O_{\delta, \epsilon} (X^{\frac{1}{2}} y^{\frac{23}{12}+\epsilon} (\log y)^2)$ in aggregate.

In Section~\ref{ssec:Wsff_residue_discrete} we then consider the contribution from the discrete spectra of $Z(s, 0, f\times f)$.
Under the assumption $y \ll X^{1-\delta}$ with $\delta > \frac{1}{2}$, the work of~\cite{HoffsteinHulse13} allows us to bound the discrete spectrum by $O_{\delta}(X^{\frac{1}{2}} y^{2} (\log y)^{\frac{1}{2}})$.
Note that this estimate is slightly larger than the bound within the continuous spectrum.

\begin{lemma}\label{lem:Wsff_Mellin_Barnes_Residue_upper_bound}
  \begin{equation*}
    \frac{1}{2\pi i} \int_{(\sigma_s)} \frac{W(s-1; f,f)}{s+k-2} \exp\left(\frac{\pi s^2}{y^2}\right) \frac{X^s}{y}ds \ll_{\delta} \frac{X^{\frac{3}{2}}}{y} + X^{\frac{1}{2}} y^{2}(\log y)^{\frac{1}{2}}.
  \end{equation*}
\end{lemma}
\begin{proof}
  See Section~\ref{sec:Wsff_Mellin_Barnes_Residue}.
\end{proof}

Combining~\eqref{eq:goal_inequality} with the three bounds from Lemmas~\ref{lem:Wsff_upper_bound},~\ref{lem:MellinBarnesWsff_upper_bound}, and~\ref{lem:Wsff_Mellin_Barnes_Residue_upper_bound}, we get the following theorem.

\begin{theorem}\label{thm:main_theorem}
  For $y \ll X^{1 - \delta}$ with $\delta > \tfrac{1}{2}$,
  \begin{equation}\label{eq:main_theorem1}
    \sum_{\lvert n - X \rvert < X/y} \frac{\lvert S_f(n) \rvert^2}{n^{k-1}} \ll_{\delta} \frac{X^{\frac{3}{2}}}{y} + X^{\frac{1}{2}} y^{2} (\log y)^{\frac{1}{2}}.
  \end{equation}
  Choosing $y = X^{\frac{1}{3}} (\log X)^{-\frac{1}{6}}$, we see that
  \begin{equation*}%\label{eq:main_theorem2}
    \frac{1}{X^{\frac{2}{3}} (\log X)^{\frac{1}{6}}} \sum_{\lvert n - X \rvert < X^{\frac{2}{3}}(\log X)^{\frac{1}{6}}} \frac{\lvert S_f(n) \rvert^2}{n^{k-1}} \ll X^{\frac{1}{2}}.
  \end{equation*}
\end{theorem}

\begin{remark}\label{rmk:main_theorem}
  The Classical Conjecture for intervals of length $X/y$ holds for any $y$ such that the first error summand in~\eqref{eq:main_theorem1} dominates.
  Our error bounds are thus optimized when $y(\log y)^{\frac{1}{6}}=X^{\frac{1}{3}}$.
  Unfortunately, the function $y \mapsto y(\log y)^{\frac{1}{6}}$ does not have an inverse expressible in common functions; nevertheless, the minimal interval is well-approximated by
  $\lvert n - X \rvert < X^{\frac{2}{3}}(\log X)^{\frac{1}{6}}$.
\end{remark}

An application of summation by parts yields Theorem~\ref{thm:classical_conjecture_in_twothird_intervals} as stated in the introduction, which shows that the Classical Conjecture is true on average over short intervals of length $X^{\frac{2}{3}}(\log X)^{\frac{1}{6}}$.
In addition, we see that the Hafner-Ivi\'{c} bound, $S_f(n) \ll n^{\frac{k-1}{2} + \frac{1}{3}}$, can be improved in the mean square aspect over intervals of length slightly less than $X^{\frac{2}{3}}$.

\begin{corollary}
  For $\delta \in (\frac{1}{2}, \frac{2}{3}]$,
  \begin{equation*}
    \frac{1}{X^\delta} \sum_{\lvert n - X \rvert< X^\delta} \lvert S_f(n) \rvert^2 \ll_{\delta}X^{k + \frac{3}{2} - 3\delta} (\log X)^{\frac{1}{2}}.
  \end{equation*}
  This improves upon the Hafner-Ivi\'{c} bound for $\delta \in (\frac{11}{18},\frac{2}{3}]$.
\end{corollary}

\section{Areas for Further Inquiry}

In this paper, the assumption $y \ll X^{\frac{1}{2}}$ is essential to much of the analysis.
There are two bounds in the paper that would need to be strengthened for Theorem~\ref{thm:classical_conjecture_in_twothird_intervals} to hold in intervals of length $X^{\frac{1}{2} + \epsilon}$.

First, better subconvexity results for $L(\tfrac{1}{2} + it, f\times f)$ would lead to better bounds for~\eqref{eq:MB_residue_jthresiduals_integral_abs} in the continuous spectrum.
Assuming the Lindel\"{o}f hypothesis for $L(s, f\times f)$, we get a bound for~\eqref{eq:MB_residue_jthresiduals_integral_abs} consistent with the Classical Conjecture in intervals of length $X^{\frac{1}{2} + \epsilon}$.

Secondly, we bound~\eqref{eq:MB_residue_discrete_pole} by absolute values and use
\begin{equation}\label{eq:hoffsteinhulse-bound}
  \sum_{\lvert t_j \rvert \sim T}  \left| \rho_j(1) \langle \lvert f \rvert^2 \Im(\cdot)^k, \mu_j \rangle \right|  \ll T^{k + 1} (\log T)^{\frac{1}{2}},
\end{equation}
from Proposition~4.3 of~\cite{HoffsteinHulse13}, to estimate the contribution from the discrete spectrum in Section~\ref{ssec:Wsff_residue_discrete}.
There are good reasons to suppose that~\eqref{eq:hoffsteinhulse-bound} is of the correct size, but in the process of taking absolute values we ignore oscillatory behaviour.
We suspect it is possible to better take advantage of the oscillatory behaviour of the summands in~\eqref{eq:MB_residue_discrete_pole} to lower the resulting bound.

Progress towards the Classical Conjecture, such as in~\cite{Deligne},~\cite{HafnerIvic89}, and~\cite{rankin1990sums}, has historically focused on full-integral weight cusp forms.
It is likely that this restriction is unnecessary and that the analogue of the Classical Conjecture will hold for half-integral weight cusp forms as well.
The methods in~\cite{hkldw} and this paper, with some small modifications, apply to the case of half-integral weight cusp forms.

However, there is one important difference.
The proof of~\eqref{eq:hoffsteinhulse-bound} in~\cite{HoffsteinHulse13} relies heavily on $f$ being of full-integral weight.
A version of~\eqref{eq:hoffsteinhulse-bound} has been proven in the half-integral case in~\cite{mehmet2015}, which takes the form
\begin{equation} \label{eq:kiral-bound}
  \sum_{\lvert t_j \rvert \sim T} \left \lvert \rho_j(1) \langle \lvert f \rvert^2 \Im(\cdot)^k, \mu_j \rangle \right \rvert \ll T^{2k + \frac{5}{2} + \epsilon}.
\end{equation}
Direct application of \eqref{eq:kiral-bound} along the lines of this paper leads to estimates for the discrete spectrum that deteriorate as the weight $k$ increases.
With a bound more similar to~\eqref{eq:hoffsteinhulse-bound}, one could produce analogous results for half-integral weight forms in short intervals of length uniform in $k$.

\section{Properties of the Integral Transforms}\label{sec:integral_transform}

In this section, we prove the Mellin integral transform~\eqref{eq:integral_transform} and a generalization of that identity for later use in Section~\ref{sec:Wsff_Mellin_Barnes_Integral}.

\begin{lemma}
  \begin{equation}\label{eq:integral_transform_later_section}
    \frac{1}{2\pi i} \int_{(2)} \exp\left(\frac{\pi s^2}{y^2}\right) \frac{X^s}{y}ds = \frac{1}{2\pi} \exp\left(-\frac{y^2 \log^2 X}{4\pi}\right).
  \end{equation}
\end{lemma}

\begin{proof}
  Write $X^s = e^{s\log X}$ and complete the square in the exponents.
  Since the integrand is entire and the integral is absolutely convergent, we may perform a change of variables $s \mapsto s-y^2 \log X/2\pi$ and shift the line of integration back to the imaginary axis.
  This yields
  \begin{equation*}
    \frac{1}{2\pi i} \exp\left( - \frac{y^2 \log^2 X}{4\pi}\right) \int_{(0)} e^{\pi s^2/y^2} \frac{ds}{y}.
  \end{equation*}
  The change of variables $s \mapsto isy$ transforms the integral into the standard Gaussian, completing the proof.
\end{proof}

Since the integral is absolutely convergent, we can produce additional integral identities by differentiating with respect to $X$ under the integral sign.
In particular, for non-negative integers $m$ and $\ell$, we multiply both sides of~\eqref{eq:integral_transform_later_section} by $X^{m+\ell-1}$, differentiate with respect to $X$ a total of $m$ times, and multiply by $X^{1-\ell}$ to get the following.
\begin{lemma}\label{lem:transform_general_j}
\begin{align*}
  \frac{1}{2\pi i} \int_{(2)} &\frac{\Gamma(s+m+\ell)}{\Gamma(s+\ell)} \exp\!\left(\frac{\pi s^2}{y^2}\right) \frac{X^s}{y} ds \\
	&= X^{1-\ell}\frac{d^m}{dX^m} \left(\frac{X^{m+\ell-1}}{2\pi} \exp\!\left( - \frac{y^2 \log^2 X}{4\pi }\right)\right).
\end{align*}
By induction, it follows that the right-hand side above is
\begin{equation}\label{eq:j_deriv_integral_bound}
  O_{m,\ell}\left( \bigg( y + y^2 \lvert \log X \rvert \bigg)^m \exp\left(-\frac{y^2 \log^2 X}{4\pi}\right) \right),
\end{equation}
where $O_{m,\ell}$ indicates that the implicit constant depends on $m$ and $\ell$.
\end{lemma}

\section{Bounding the Contribution from $W(s; f,f)$}\label{sec:Wsff}

In this section we prove Lemma~\ref{lem:Wsff_upper_bound}.
By separating the cases $h = 0$ and $m=0$ and performing a small inclusion-exclusion correction, we collect $W(s; f,f)$ into a more standard Dirichlet series:
\begin{align}
  W(s; f,f) &= \sum_{n \geq 1} \frac{a(n)\overline{S_f(n)} + \overline{a(n)}S_f(n) - \lvert a(n) \rvert^2}{n^{s + k - 1}}. \label{eq:W_as_Dirichlet_series}
\end{align}
Denote the $n$th coefficient of this Dirichlet series by $w(n)$.
Since $a(n) \ll n^{\frac{k-1}{2} + \epsilon}$ and $S_f(n) \ll n^{\frac{k-1}{2} + \frac{1}{3}}$, by~\cite{HafnerIvic89}, we see that $w(n) \ll n^{k-1 + \frac{1}{3} + \epsilon}$. Thus the Dirichlet series for $W(s;f,f)$ converges absolutely for $\Re s > \frac{4}{3}+\epsilon$.

By applying the integral transform~\eqref{eq:integral_transform_later_section} directly to $W(s; f,f)$, we get
\begin{align*}
  &\frac{1}{2\pi i} \int_{(\sigma_s)} W(s; f,f) \exp\left(\frac{\pi s^2}{y^2}\right) \frac{X^s}{y}ds \\
  &\quad\ll \sum_{n \geq 1} n^{\frac{1}{3} + \epsilon} \exp\left( -\frac{y^2 \log^2 (\frac{X}{n})}{4\pi}\right) \ll_\epsilon \frac{X^{\frac{4}{3} + \epsilon}}{y}.
\end{align*}
This completes the proof.

\begin{remark}
  This bound uses the best-known polynomial bound on $S_f(n)$, and so is probably weaker than reality; under the Classical Conjecture, this term would contribute no more than $X^{\frac{5}{4} + \epsilon}/y$.
Nevertheless, this term is so much smaller than the others that improving the bound is not necessary for this application.
\end{remark}

\section{Bounding the Contribution from $W(s-1; f,f)/(s+k-2)$}\label{sec:Wsff_Mellin_Barnes_Residue}

In this section, we prove Lemma~\ref{lem:Wsff_Mellin_Barnes_Residue_upper_bound} by bounding
\begin{equation}\label{eq:MB_residue_base}
  \frac{1}{2\pi i} \int_{(\sigma_s)} \frac{W(s - 1; f,f)}{s + k - 2} \exp\left(\frac{\pi s^2}{y^2}\right) \frac{X^s}{y} ds.
\end{equation}
Recall that $W(s; f,f) = L(s, f\times f)\zeta(2s)^{-1} + Z(s, 0, f\times f)$ and that $Z(s, 0, f\times f)$ splits into discrete and continuous components in its meromorphic continuation as given in Proposition~\ref{prop:Z_meromorphic_continuation}.

In Section~\ref{ssec:Wsff_residue_continuous} we consider the part of~\eqref{eq:MB_residue_base} corresponding to $L(s, f\times f)\zeta(2s)^{-1}$ and the continuous spectrum of $Z(s, 0, f\times f)$.
In Section~\ref{ssec:Wsff_residue_discrete} we consider the part of~\eqref{eq:MB_residue_base} corresponding to the discrete spectrum of $Z(s, 0, f\times f)$.

\subsection{Rankin-Selberg Convolution and Continuous Spectrum}\label{ssec:Wsff_residue_continuous}

Accounting for the residual terms that appear through meromorphic continuation, the contribution to~\eqref{eq:MB_residue_base} from the continuous spectrum of $Z(s-1, 0, f\times f)/(s+k-2)$ is given by
\begin{equation}\label{eq:MB_residue_continuous_base}
  \begin{split}
    &\frac{(4\pi)^k}{2 (2\pi i)^2} \int_{(\sigma_s)}\int_{(0)} \bigg[ \frac{\Gamma(s - \frac{3}{2} + z) \Gamma(s - \frac{3}{2} - z)}{\Gamma(s-1) \Gamma(s + k - 1)} \frac{\zeta(s - \frac{3}{2} + z) \zeta(s - \frac{3}{2} - z)}{\zeta^*(1 + 2z)} \\
  &\quad \times \langle \mathcal{V}_{f,f}, \overline{E(\cdot, \tfrac{1}{2} - z)} \rangle  + \sum_{m = 0}^{\lfloor \frac{5}{2} - \Re s \rfloor} \frac{\rho_{\frac{3}{2} - m}(s-1)}{s + k - 2} \bigg] \exp\left(\frac{\pi s^2}{y^2}\right) \frac{X^s}{y} dzds.
  \end{split}
\end{equation}
Above, we are using the meromorphic continuation in~\eqref{eq:meromorphic_continuation}.
We refer to the first term in the brackets as the \emph{main integrand}; the remaining terms are the \emph{residual terms}.
For $m \geq 0$, the $m$th residual term $\rho_{\frac{3}{2} - m}(s-1)$ appears only when $\Re (s - 1) < \frac{3}{2} - m$.
So the first residual term appears in the meromorphic continuation when $\Re s < \frac{5}{2}$.

Now we move the line of $s$-integration in the negative direction.
Fix $\delta > \frac{1}{2}$ and suppose $y \ll X^{1 - \delta}$.
We shift $\sigma_2$ to $-M_\delta$, where $M_\delta > 0$ is a constant depending on $\delta$, which we specify later.
During this shift, approximately $M_\delta / 2$ residual terms $\rho_{\frac{3}{2} - m}(s-1)/(s + k - 2)$ appear in the meromorphic continuation as indicated in~\eqref{eq:MB_residue_continuous_base} above.
We separate each residual term and bound them separately from the shifted main integral.

\subsubsection{Shifted Main Integral}

%Recall our assumption that $y \ll X^{1 - \delta}$ for a fixed $\delta > \tfrac{1}{2}$.
%In~\eqref{eq:MB_residue_continuous_base}, shift $\sigma_s$ to $-M_\delta$.
The identity for the Rankin-Selberg convolution, as in Section~1.6 of~\cite{Bump98}, implies
\begin{equation}\label{eq:rankin-selberg-identity}
  \langle \mathcal{V}_{f,f}, \overline{E(\cdot, s)}\rangle = 2 \cdot\frac{\Gamma(s + k - 1)}{(4\pi)^{s + k - 1}} \frac{L(s, f\times f)}{\zeta(2s)}.
\end{equation}
We substitute \eqref{eq:rankin-selberg-identity} into the shifted main integrand of~\eqref{eq:MB_residue_continuous_base} and apply the asymmetric functional equation of the Riemann zeta function,
\begin{equation}\label{eq:zeta_asymmetric_functional_equation}
  2\zeta(z)\Gamma(z) = (2\pi)^{z} \sec(\tfrac{\pi z}{2}) \zeta(1 - z),
\end{equation}
to the zeta functions in the numerator.
Since $\Re s < 0$, the reflected zeta functions are in the domain of convergence and are uniformly bounded.

Bounding the integrand by absolute values, we recall that $1/\zeta(1 \pm 2z) \ll \lvert \log z \rvert$ for $\Re z = 0$ (see~\cite[3.11.10]{titchmarsh1986theory}), and apply Stirling's approximation to see that the integrand in~\eqref{eq:MB_residue_continuous_base} is
\begin{equation*}%\label{eq:MB_residue_first_integrand}
  O_{\delta} \left( X^{-M_\delta} \frac{\lvert z \rvert^{k-1} \lvert \log z \rvert }{\lvert s \rvert^{k - 3 - 2M_\delta}} \mathcal{E}(s,z) \lvert L(\tfrac{1}{2} - z, f\times f) \rvert\right),
\end{equation*}
where $\mathcal{E}(s,z)$ are the collected exponential contributions
\begin{equation*}
  \mathcal{E}(s,z) = \frac{\exp(\pi \lvert \Im s \rvert)}{\exp(\pi \max\{ \lvert \Im s \rvert, \lvert \Im z \rvert\})} \exp\!\left( - \frac{\pi \Im^2 s}{y^2}\right).
\end{equation*}

To bound the integral, we split the line of $z$-integration into two regions.
When $\lvert \Im z \rvert < \lvert \Im s \rvert$, we bound~\eqref{eq:MB_residue_continuous_base} by
\begin{equation*}
  X^{-M_\delta} \int_{(-M_\delta)} \lvert s \rvert^{2 + 2M_\delta} \log \lvert s \rvert \exp\left(\frac{- \pi \Im^2 s}{y^2}\right)\int_{-\lvert \Im s \rvert}^{\lvert \Im s \rvert} \lvert L(\tfrac{1}{2} - it, f\times f) \rvert \frac{dt ds}{y}.
\end{equation*}
%The $t$ integral above can be bounded by $O_{\epsilon '}(\lvert \Im s \rvert^{2 + \epsilon'})$ by the convexity bound for $L(s, f\times f)$.
The integral in the $t$-variable is quickly seen to be $O(\lvert \Im s \rvert^{2})$ using Soundarajan's general weak subconvexity result~\cite{soundararajan2010weak}.
Soundarajan's result indicates a further small fractional power of $\log \lvert \Im s \rvert$ savings, which we neglect for simplicity.
The remaining integral is
\begin{equation}\label{eq:MB_residue_firstintegrand_boundzsmall}
  O_{\delta} \left( X^{-M_\delta} y^{4 + 2M_\delta} \log y \right).
\end{equation}

The contribution from the remaining region, $\lvert \Im z \rvert > \lvert \Im s \rvert$, is easier to bound as $\mathcal{E}(s,z)$ has exponential decay in both variables. In this case, Soundarajan's subconvexity bound and a similar argument gives that the contribution from this region is
\begin{equation}\label{eq:MB_residue_firstintegrand_boundzlarge}
  O_{\delta} \left( X^{-M_\delta} y^{3 + 2M_\delta} \log y\right).
\end{equation}

By choosing $M_\delta$ large, the two bounds~\eqref{eq:MB_residue_firstintegrand_boundzsmall} and~\eqref{eq:MB_residue_firstintegrand_boundzlarge} can be made arbitrarily small for $X$ large enough, under the assumption $y \ll X^{1-\delta}$ and $\delta > \tfrac{1}{2}$.
While the implicit constant depending on $M_\delta$ might be very large, it is independent of $X$ and $y$.
In particular, if we choose $M_\delta > 4(2\delta - 1)^{-1} (1-\delta)$, the two bounds \eqref{eq:MB_residue_firstintegrand_boundzsmall} and \eqref{eq:MB_residue_firstintegrand_boundzlarge} are $O_\delta(\log y)$ and non-dominant.

\subsubsection{The First Residual Term}

The first residual term, $\rho_{\frac{3}{2}}(s-1)/(s + k - 2)$, of $Z(s-1,0,f\times f)/(s + k - 2)$ appears for $\Re s < \frac{5}{2}$.
So beginning with $\sigma_s \in (2, \tfrac{5}{2})$ and substituting the definition~\eqref{eq:zeta_residual} of $\rho_{\frac{3}{2}}$ in the first residual term in~\eqref{eq:MB_residue_continuous_base}, we are led to consider the integral
\begin{equation*}
  \frac{(4\pi)^k}{4 \pi i} \int_{(\sigma_s)} \frac{\zeta(2s - 4)\Gamma(2s - 4) \langle \mathcal{V}_{f,f}, \overline{E(\cdot, 3 - s)} \rangle}{\Gamma(s-1) \Gamma(s + k - 1)\zeta^*(2s - 4)} \exp\!\left(\frac{\pi s^2}{y^2}\right) \frac{X^s}{y}ds.
\end{equation*}
By expressing the inner product in terms of the Rankin-Selberg convolution via~\eqref{eq:rankin-selberg-identity}, cancelling zeta functions, and applying the Gauss duplication formula to the ratio $\Gamma(2s - 4)/\Gamma(s-2)$, we are able to rewrite the above integral as
\begin{equation}\label{eq:MB_residue_firstresidual_base}
  \frac{1}{2\pi i} \int_{(\sigma_s)} \frac{\Gamma(s - \frac{3}{2}) \Gamma(k + 2 - s) L(3 - s, f\times f)}{(4\pi)^{\frac{9}{2} - 2s} \Gamma(s-1) \Gamma(s + k - 1) \zeta(6 - 2s)} \exp\!\left(\frac{\pi s^2}{y^2}\right) \frac{X^s}{y} ds.
\end{equation}

Shifting $\sigma_s$ into the interval $(\tfrac{3}{2}, 2)$ passes over a pole at $s = 2$.
It is shown in~\cite{hkldw} that the pole at $s = 2$ from $L(s-1, f\times f) \zeta(2s - 2)^{-1}$ is cancelled by the pole at $s = 2$ of $Z(s-1, 0, f\times f) $.

We then shift $\sigma_s$ to $-M_\delta$, passing approximately $M_\delta$ poles at points of the form $s = \tfrac{3}{2} - m$, with $m \in \mathbb{Z}_{\geq 0}$.
The residues at each pole are bounded in $X$ and $y$ by the pole with the largest residue, which occurs at $s = \tfrac{3}{2}$.
The contribution towards \eqref{eq:MB_residue_firstresidual_base} from these polar terms is therefore
\begin{equation}\label{eq:MV_residual_firstresidual_bound_poles}
  O_\delta \left( \frac{X^{\frac{3}{2}}}{y}\right).
\end{equation}
% It's interesting to note that the implicit constants are unbelievably huge in M, something like e^{M^2} (M!)^2.

\begin{remark}
  This is one of the two balancing terms that produces the final bound in our main theorem.
\end{remark}

It remains to bound the shifted first-residual integral~\eqref{eq:MB_residue_firstresidual_base}.
As $\Re s < 2$, we note that the $L$-function and $\zeta(6 - 2s)$ are in their half-planes of convergence and are uniformly bounded.
Writing $s = -M_\delta + it$, bounding the integrand by absolute values, and applying Stirling's approximation gives that~\eqref{eq:MB_residue_continuous_base} is
\begin{equation}\label{eq:MB_residue_firstresidual_bound_shiftedintegral}
  O_\delta \left( X^{-M_\delta} \int_{-\infty}^\infty \lvert t \rvert^{\frac{5}{2} + 2M_\delta} \exp\left( - \frac{\pi t^2}{y^2}\right) \frac{dt}{y} \right) \ll_\delta X^{-M_\delta} y^{\frac{5}{2} + 2M_\delta}.
\end{equation}
Just as in the analysis of the bounds~\eqref{eq:MB_residue_firstintegrand_boundzsmall} and~\eqref{eq:MB_residue_firstintegrand_boundzlarge}, we can make~\eqref{eq:MB_residue_firstresidual_bound_shiftedintegral} arbitrarily small as $X$ tends to infinity by choosing $M_\delta$ sufficiently large, provided that $y \ll X^{1 - \delta}$ with $\delta > \frac{1}{2}$.
In particular, if we choose $M_\delta > 3(2\delta - 1)^{-1}(1-\delta)$, then~\eqref{eq:MB_residue_firstresidual_bound_shiftedintegral} is non-dominant.

\subsubsection{Further Residual Terms}

In~\cite{hkldw}, it is shown that the second residual term of $Z(s-1, 0, f\times f)$, which we denote as $\rho_{\frac{1}{2}}(s-1)$ and which appears for $\Re s < \frac{3}{2}$, identically cancels with $L(s - 1, f\times f)\zeta(2s - 2)^{-1}$.
In other words, the second residual term and the Rankin-Selberg convolution $L$-function completely cancel for $\Re s < \frac{3}{2}$.

Further residual terms can all be handled systematically.
For each integer $m$ with $1 \leq m \leq M_\delta + \frac{5}{2}$, a further residual term $\rho_{\frac{3}{2} - m}(s)/(s+k-2)$ is introduced as part of the meromorphic continuation of $Z(s-1, 0, f\times f)/(s+k-2)$.
Substituting the expression for $\rho_{\frac{3}{2} - m}(s)$ given in~\eqref{eq:jth_residual} into the integral~\eqref{eq:MB_residue_continuous_base}, we see that the integrals corresponding to these residual terms take the form
\begin{align}\label{eq:MB_residue_mthresiduals_integral_base}
  \begin{split}
  \frac{2(-1)^m (4\pi)^k \zeta(-m)}{m! \,2\pi i} \int_{(\sigma_m)}& \frac{\zeta(2s + m - 3) \Gamma(2s + m - 3)}{\Gamma(s - 1) \Gamma(s + k - 1) \zeta^*(4 - 2m - 2s)} \\
 &\times \langle \mathcal{V}_{f,f}, \overline{E(\cdot, s + m - 1)} \rangle\exp\left(\frac{\pi s^2}{y^2}\right) \frac{X^s}{y}ds,
  \end{split}
\end{align}
where $\sigma_m = \frac{3}{2} - m$ for each $m$.
Notice that the apparent pole from $\Gamma(2s + m - 3)$ cancels with the pole of $\zeta^*(4 - 2m - 2s)$ in the denominator, so that each integrand is actually holomorphic along the lines of integration.

Without assuming progress towards the Riemann Hypothesis, it is not possible to shift the lines of integration further in the negative direction, as the uncompleted Eisenstein series $E(\cdot, s + m - 1)$ has poles at $s = 1 + \frac{\rho}{2} - m$ for each nontrivial zero $\rho$ of the zeta function.

We therefore estimate each integral directly, writing the inner product in terms of Rankin-Selberg convolutions via~\eqref{eq:rankin-selberg-identity}, reflecting the zeta function in the numerator with its functional equation, bounding the integrand by absolute values, and applying Stirling's formula to bound each integral~\eqref{eq:MB_residue_mthresiduals_integral_base} by
\begin{equation}\label{eq:MB_residue_jthresiduals_integral_abs}
  X^{\frac{3}{2} - m} \int_{-\infty}^\infty (1 + \lvert t \rvert)^{2m - 1} \frac{\lvert L(\frac{1}{2} + it, f\times f) \rvert}{\lvert \zeta(1 + 2it) \rvert^2} \exp\left(- \frac{\pi t^2}{y^2}\right) \frac{dt}{y}.
\end{equation}

Recall that $L(s, f\times f)$ factors as the product $L(\text{Sym}^2 f, s) \zeta(s)$, where $L(\text{Sym}^2 f, s)$ denotes the $\text{GL}(3)$ symmetric square $L$-function.
In~\cite{li2011bounds}, Li shows that $\lvert L(\text{Sym}^2 f, \tfrac{1}{2} + it)\rvert \ll_\epsilon t^{\frac{2}{3} + \epsilon}$.
Combined with the convexity bound $\lvert \zeta(\tfrac{1}{2} + it) \rvert \ll_\epsilon t^{\frac{1}{4} + \epsilon}$ and de la Vall\'{e}e-Poussin's estimate, $1/\zeta(1 + it) \ll \lvert \log t \rvert$, to show that the $m$th integral~\eqref{eq:MB_residue_jthresiduals_integral_abs} is
\begin{equation}\label{eq:MB_residue_jthresiduals_bounds}
  O_{\epsilon,m} (X^{\frac{3}{2} - m} y^{2m - \frac{1}{12} + \epsilon} (\log y)^2).
\end{equation}

Using once again the assumptions that $y \ll X^{1 - \delta}$ and $\delta > \tfrac{1}{2}$, we see that each term of the form~\eqref{eq:MB_residue_jthresiduals_bounds} is bounded above by the instance of~\eqref{eq:MB_residue_jthresiduals_bounds} with $m=1$.
We have approximately $M_\delta$ integrals, which jointly contribute
\begin{equation}\label{eq:MB_residue_jthresiduals_bound_final}
  O_{\delta} \left(X^{\frac{1}{2}} y^{\frac{23}{12}} (\log y)^2\right).
\end{equation}

\begin{remark}
  This nearly matches the bound from~\ref{eq:MB_residue_discrete_large_bound}, the second of the two balancing terms that produce the upper bound in Theorem~\ref{thm:main_theorem}.
  Note that we have used the partial subconvexity bounds from~\cite{li2011bounds}, though by the Lindel\"{o}f hypothesis we should expect the vastly better bound $L(\tfrac{1}{2} + it, f\times f) \ll t^{\epsilon}$.
\end{remark}

\subsection{Discrete Spectrum}\label{ssec:Wsff_residue_discrete}

Beginning with $\sigma_s > \frac{5}{2}$, the contribution towards~\eqref{eq:MB_residue_base} from the discrete spectum of $Z(s-1, 0, f\times f)/(s+k-2)$ is given by
\begin{align}\label{eq:MB_residue_discrete_base}
  \begin{split}
    \frac{(4\pi)^k}{2} \sum_j \rho_j(1) \langle \mathcal{V}_{f,f}, \mu_j \rangle \int_{(\sigma_s)} &\frac{\Gamma(s - \frac{3}{2} + it_j) \Gamma(s - \frac{3}{2} + it_j)}{\Gamma(s - 1) \Gamma(s + k - 1)} \\
    &\times L(s - \tfrac{3}{2}, \mu_j) \exp\!\left(\frac{\pi s^2}{y^2}\right) \frac{X^s}{y} ds.
  \end{split}
\end{align}
Above, we are using the meromorphic continuation from Proposition~\ref{prop:Z_meromorphic_continuation}.
A priori, the sum is over the complete orthonormal basis of Maass eigenforms, but we note that $\langle \mathcal{V}_{f,f}, \mu_j\rangle = 0$ whenever $\mu_j$ is odd, as was demonstrated in \cite{hkldw} as a consequence of Watson's triple product formula \cite{watson2008rankin}.
So we may assume that each $\mu_j$ is even in the spectral sum above.

For an even Maass form $\mu_j$, its corresponding $L$-function satisfies the functional equation~\cite[Chap~3]{Goldfeld2006automorphic}
\begin{equation}\label{eq:maassform_functional}
  \Lambda(s) := \pi^{-s} \Gamma\left(\frac{s + it_j}{2}\right) \Gamma\left( \frac{s - it_j}{2}\right) L(s, \mu_j) = \Lambda(1-s),
\end{equation}
indicating that $L(s, \mu_j)$ has trivial zeroes at $s=\pm it_j$.
So we shift the line of integration in \eqref{eq:MB_residue_discrete_base} to $\sigma_s = \frac{3}{2} - \epsilon$ without encountering any poles.

Using the Gauss duplication formula, one can split the product of two gamma factors in the numerator of~\eqref{eq:MB_residue_discrete_base} into a product of four gamma factors, two of which are the necessary gamma functions to complete $L(s - \frac{3}{2}, \mu_j)$.
Applying the functional equation~\eqref{eq:maassform_functional} and the gamma reflection property, $\Gamma(z) \Gamma(1-z) = \pi \csc(\pi z)$, we find that all of these gamma factors cancel out.

Up to constants and factors of $2$ and $\pi$, which do not contribute to the upper bound, we are able to rewrite~\eqref{eq:MB_residue_discrete_base} as
\begin{align}\label{eq:MB_residue_discrete_base_rewrite}
  \begin{split}
    \sum_j \rho_j(1) \langle \mathcal{V}_{f,f}, \mu_j \rangle\int_{(\sigma_s)} &\frac{\csc(\frac{\pi}{2}(s - \frac{1}{2} + it_j)) \csc(\frac{\pi}{2}(s - \frac{1}{2} - it_j))}{\Gamma(s-1) \Gamma(s + k - 1)} \\
    &\,\times L(\tfrac{5}{2} - s, \mu_j) (2\pi)^{2s} \exp\!\left(\frac{\pi s^2}{y^2}\right) \frac{X^s}{y}ds.
  \end{split}
\end{align}
We then shift the line of integration from $\sigma_s = \frac{3}{2} - \epsilon$ to $\sigma_s = -M_\delta$, where $M_\delta$ is a constant depending on $\delta$ which we'll still specify later.
The integrand of~\eqref{eq:MB_residue_discrete_base_rewrite} has poles at $s = \tfrac{1}{2} - 2m \pm it_j$ for $m \in \mathbb{Z}_{\geq 0}$.
The residues at these poles take the form
\begin{align}\label{eq:MB_residue_discrete_pole}
  \begin{split}
    \frac{1}{y} \sum_j &\rho_j(1) \langle \mathcal{V}_{f,f}, \mu_j \rangle \frac{\csc(\pi(-m \pm it_j)) L(2 + 2m \mp it_j, \mu_j)}{\Gamma(-\frac{1}{2} - 2m \pm it_j) \Gamma(k - 2m - \frac{1}{2} \pm it_j)} \\
                       &\times X^{\frac{1}{2} - 2m \pm it_j} (2\pi)^{\pm 2it_j} \exp\!\left( \frac{\pi (\frac{1}{2} - 2m \pm it_j)^2}{y^2}\right).
  \end{split}
\end{align}
Bounding the summands by absolute values and applying Stirling's approximation, we bound~\eqref{eq:MB_residue_discrete_pole} above by
\begin{equation}\label{eq:MB_residue_discrete_pole_initialbound}
  \frac{X^{\frac{1}{2} - 2m}}{y} \sum_j \left\lvert \rho_j(1) \langle \mathcal{V}_{f,f}, \mu_j \rangle\right\rvert \lvert t_j \rvert^{2 + 4m - k} \exp\!\left( - \frac{\pi t_j^2}{y^2}\right).
\end{equation}
In Proposition~4.3 of~\cite{HoffsteinHulse13}, it is proved that
\begin{equation}\label{eq:hoffsteinhulse_spectralbound}
  \sum_{\frac{T}{2} \leq t_j \leq 2T} \lvert \rho_j(1) \langle \lvert f \rvert^2 \Im(\cdot)^k, \mu_j \rangle \rvert \ll T^{k + 1} (\log T)^{1/2}.
\end{equation}
Weighting individual terms and noting that $\langle V_{f,f}, \mu_j \rangle = 2 \langle \vert f \vert^2 \Im(\cdot)^k,\mu_j \rangle$ when $\mu_j$ is even, it follows that
\begin{align*}
\sum_{\frac{T}{2} \leq t_j \leq 2T} \lvert \rho_j(1) \langle \mathcal{V}_{f,f}, \mu_j \rangle \rvert  &\vert t_j \vert^{2+4m-k} \exp\left(-\frac{\pi tj^2}{y^2}\right) \\
&\ll T^{3 + 4m} (\log T)^{1/2} \exp\left(-\frac{\pi (T/2)^2}{y^2}\right).
\end{align*}
Summing dyadically, we extend this bound to all $t_j$ to show that~\eqref{eq:MB_residue_discrete_pole_initialbound} is
\begin{equation}\label{eq:MB_residue_discrete_pole_middlebound}
  O_{m} \left( X^{\frac{1}{2} - 2m} y^{2 + 4m} (\log y)^{\frac{1}{2}} \right).
\end{equation}
There are about $M_\delta / 2$ pairs of poles, each contributing residues which are bounded above by the $m = 0$  term in~\eqref{eq:MB_residue_discrete_pole_middlebound}.
It follows that the total contribution of these poles is
\begin{equation}\label{eq:MB_residue_discrete_large_bound}
  O_{\delta} (X^{\frac{1}{2}} y^{2} (\log y)^{\frac{1}{2}}).
\end{equation}

\begin{remark}
  This estimate is the second of the two balancing terms in our final bound.
  We used absolute values to go from~\eqref{eq:MB_residue_discrete_pole} to~\eqref{eq:MB_residue_discrete_pole_initialbound} in order to directly apply~\eqref{eq:hoffsteinhulse_spectralbound}, but this ignores the oscillatory behaviour of the original summands.
\end{remark}

It remains to bound the shifted discrete component integral~\eqref{eq:MB_residue_discrete_base_rewrite} on the line $\sigma_s = -M_\delta$.
Write $s = -M_\delta + it$, and split the integral and sum into two regions depending on the relative size of $\lvert t_j \rvert$ and $\lvert \Im s \rvert$.

When $\lvert t_j \rvert < \lvert t \rvert$, bounding the integrand by absolute values and applying Stirling's approximation leads to the upper bound
\begin{equation}\label{eq:MB_residue_discrete_integral_tjsmall}
  \ll_\delta X^{-M_\delta} \! \int_{0}^\infty \sum_{\lvert t_j \rvert < \lvert t \rvert} \left \lvert \rho_j(1) \langle \lvert f \rvert^2 \Im(\cdot)^k, \mu_j \rangle \right \rvert t^{2M_\delta - k + 3} \exp\!\left( -\frac{\pi t^2}{y^2}\right) \frac{dt}{y}.
\end{equation}
Applying~\eqref{eq:hoffsteinhulse_spectralbound} shows that the inner spectral sum is $O_\epsilon \left(t^{k + 1 + \epsilon}\right)$.
Inserting this bound gives the integral
\begin{equation*}
  X^{-M_\delta} \int_{0}^\infty t^{2M_\delta + 4 + \epsilon} \exp \left( - \frac{\pi t^2}{y^2} \right) \frac{dt}{y}.
\end{equation*}
Performing the change of variables $t \mapsto ty$ removes $y$-dependence from the integral, and we see that~\eqref{eq:MB_residue_discrete_integral_tjsmall} is
\begin{equation}\label{eq:MB_residue_discrete_integral_tjsmall_bound}
  O_{\epsilon, \delta} (X^{-M_\delta}y^{4 + 2M_\delta + \epsilon}).
\end{equation}

Similarly, when $\lvert t \rvert < \lvert t_j \rvert$, we have the upper bound
\begin{align}\label{eq:MB_residue_discrete_integral_tjlarge}
  \begin{split}
  \ll_\delta X^{-M_\delta} \int_0^\infty &\left( \sum_{\lvert t_j \rvert > t} \left \lvert \rho_j(1) \langle \lvert f \rvert^2 \Im(\cdot)^k, \mu_j \rangle \right \rvert \exp\left( \frac{\pi (\lvert t \rvert - \lvert t_j \rvert)}{2}\right) \right) \\
  &\times t^{2M_\delta - k + 3} \exp\left( - \frac{\pi t^2}{y^2}\right) \frac{dt}{y}.
  \end{split}
\end{align}
Using the exponential decay and dyadic summation with~\eqref{eq:hoffsteinhulse_spectralbound}, one can show that the spectral parenthetical above is $O_\epsilon(t^{k + 1 + \epsilon})$. %Sweet phrase: spectral parenthetical. (>")> Thank you <("<)
Inserting this bound into~\eqref{eq:MB_residue_discrete_integral_tjlarge} and performing the change of variables $t \mapsto ty$ decouples $X$ and $y$ from the integral, and we have that~\eqref{eq:MB_residue_discrete_integral_tjlarge} is also
\begin{equation}\label{eq:MB_residue_discrete_integral_tjlarge_bound}
  O_{\epsilon,\delta} (X^{-M_\delta}y^{4 + 2M_\delta + \epsilon}).
\end{equation}

In both cases, the assumptions $y \ll X^{1-\delta}$ and $\delta > \frac{1}{2}$ indicate that~\eqref{eq:MB_residue_discrete_integral_tjsmall_bound} and~\eqref{eq:MB_residue_discrete_integral_tjlarge_bound} can be made arbitrarily small as $X$ tends to infinity.
In particular, if $M_\delta > 5(2\delta - 1)^{-1} (1 - \delta)$, then these terms are non-dominant.

Thus the total contribution from the discrete components of $Z(s, 0, f\times f)$ is
\begin{equation}
  O_{\delta} (X^{\frac{1}{2}} y^{2} (\log y)^{1/2}).
\end{equation}

\section{Bounding the Contribution from the Mellin-Barnes Integral~\eqref{line3:decomposition}}\label{sec:Wsff_Mellin_Barnes_Integral}

In this section, we prove Lemma~\ref{lem:MellinBarnesWsff_upper_bound} by bounding the double integral
\begin{equation}\label{eq:MB_integral_base_integral}
  \frac{1}{(2\pi i)^2} \int_{(\sigma_s)} \int_{(\sigma_z)} \! W(s-z; f,f) \zeta(z) B(s+k-1-z,z) \exp\!\left(\frac{\pi s^2}{y^2}\right) \frac{X^s}{y} dzds,
\end{equation}
in which
\begin{equation*}
  B(s,z) = \frac{\Gamma(z) \Gamma(s)}{\Gamma(s +z)}
\end{equation*}
is the Beta function, $\sigma_z = \epsilon \in (0,1)$, and $\sigma_s > \frac{3}{2} + \epsilon$.

Fix $\delta > 0$ and suppose that $y \ll X^{1 - \delta}$.
Applying the asymmetric functional equation of the Riemann zeta function~\eqref{eq:zeta_asymmetric_functional_equation} allows us to rewrite \eqref{eq:MB_integral_base_integral} as
\begin{align}\label{eq:MB_integral_shifted_base}
  \begin{split}
    \frac{1}{2 (2\pi i)^2} \int_{(\sigma_s)} \int_{(-\epsilon)} &W(s-z; f,f) \zeta(1-z)  \\
    &\times \frac{(2\pi)^z\Gamma(s + k - 1 - z)}{\cos(\tfrac{\pi z}{2}) \Gamma(s + k - 1)} \exp\left(\frac{\pi s^2}{y^2}\right) \frac{X^s}{y} dzds.
  \end{split}
\end{align}

Note that the integrand in \eqref{eq:MB_integral_shifted_base} experiences significant exponential decay in vertical strips, so we have absolute convergence away from poles.
We may therefore simultaneously shift the two lines of integration, keeping $\sigma_s = \frac{3}{2} + \sigma_z + \epsilon$ while shifting $\sigma_z$ to $-\epsilon - M_\delta$, where $M_\delta \geq 0$ is a constant depending on $\delta$ to be specified later.
This introduces several poles, the first of which comes from $\zeta(1-z)$ as $\sigma_z$ passes by $z=0$, and further poles at odd $z = -m$ for $1 \leq m \leq M_\delta$ coming from the cosine term.

The pole from $\zeta(1-z)$ at $z = 0$ has residue
\begin{equation} \label{eq:MB_trivial_first_residue_bound}
-\frac{1}{4\pi i} \int_{(\frac{3}{2}+\epsilon)} W(s;f,f) \mathrm{exp}\left(\frac{\pi s^2}{y^2}\right)\frac{X^s}{y}ds \ll_\epsilon \frac{X^{\frac{4}{3}+\epsilon}}{y}.
\end{equation}
The integral is a constant multiple of the term considered in Lemma~\ref{lem:Wsff_upper_bound}, which gives the stated upper bound.
The poles at negative odd integers $z=-m$ for $1 \leq m \leq M_\delta$ each have residue
\begin{equation} \label{eq:MB_jth_residue}
  \frac{r_m}{2\pi i} \int_{(\frac{3}{2} - m + \epsilon)} W(s+m; f,f) \frac{\Gamma(s + k - 1 + m)}{\Gamma(s + k - 1)} \exp\!\left(\frac{\pi s^2}{y^2}\right) \frac{X^s}{y} ds,
\end{equation}
in which
\begin{equation*}
  r_m = \frac{(-1)^m \zeta(1 + m)}{2(2\pi)^m}.
\end{equation*}

Note that $W(\cdot;f,f)$ lies within the half-plane of convergence of its Dirichlet series~\eqref{eq:W_as_Dirichlet_series}.
Recall that $w(n)$ denotes the $n$th coefficient of the Dirichlet series representation for $W(\cdot; f,f)$.
Expanding $W(\cdot;f,f)$ in series and applying the integral transform from Lemma~\ref{lem:transform_general_j} bounds~\eqref{eq:MB_jth_residue} by
\begin{equation}\label{eq:MB_jth_residue_bound_1}
 O_{m}\bigg( r_m \sum_{n \geq 1} \frac{w(n)}{n^{m+k-1}} (y + y^2 \vert \log(X/n)\vert)^m \exp\left(- \frac{y^2 \log^2(X/n)}{4\pi}\right)\bigg).
\end{equation}
Using the bound $w(n) \ll n^{k-1 + \frac{1}{3} + \epsilon'}$ as in Section~\ref{sec:Wsff}, we bound~\eqref{eq:MB_jth_residue_bound_1} above by
\begin{equation} \label{eq:MB_integral_comparison}
\ll_{m,\epsilon'}\int_1^\infty t^{\frac{1}{3}-m+\epsilon'}(y+y^2 \vert\log(t/X)\vert)^m \exp\left(-\frac{y^2 \log^2(t/X)}{4\pi}\right)dt.
\end{equation}
The change of variables $u=y \log(t/X)$ transforms~\eqref{eq:MB_integral_comparison} into
\begin{align*}
X^{\frac{4}{3}-m+\epsilon'}y^{m-1} \int_{0}^\infty e^{(\frac{4}{3}-m + \epsilon')\frac{u}{y}}(1+ \vert u \vert)^m \mathrm{exp}\left(-\frac{u^2}{4\pi}\right)\,du.
\end{align*}
When $m = 1$ and when $y \gg 1$ we can bound the integral by the value when $y = 1$,
\begin{equation*}
  \int_0^\infty e^{(\frac{1}{3} + \epsilon')u}(1 + \lvert u \rvert)^m \exp\left( - \frac{u^2}{4\pi} \right) du,
\end{equation*}
which converges independently of $y$.
Similarly, for $m > 1$ we can bound the integral by ignoring the exponential damping term $e^{(\frac{4}{3} - m + \epsilon')\frac{u}{y}}$,
\begin{equation*}
  \int_0^\infty (1 + \lvert u \rvert)^m \exp\left( - \frac{u^2}{4\pi} \right) du,
\end{equation*}
which also converges independently of $y$.
So we conclude that~\eqref{eq:MB_jth_residue_bound_1} is
\begin{equation*}
  O_{m, \epsilon'} \left(X^{\frac{1}{3} + \epsilon'} \left(\frac{y}{X}\right)^{m-1}\right)
\end{equation*}
for any $\epsilon' > 0$.

Under the assumption $y \ll X^{1-\delta}$, the residue~\eqref{eq:MB_jth_residue} associated to $m = 1$ dominates the other residues for large $X$.
There are at most $M_\delta / 2$ residues of the form~\eqref{eq:MB_jth_residue}, so the total contribution from these residues is
\begin{equation}\label{eq:MB_total_residue_bound}
  O_{\delta, \epsilon'} (X^{\frac{1}{3} + \epsilon'}).
\end{equation}

Now, having shifted the lines of integration so that $\sigma_s = \frac{3}{2} - M_\delta$ and $\sigma_z = -\epsilon - M_\delta$, we bound the shifted integral~\eqref{eq:MB_integral_shifted_base} in absolute value.
The Dirichlet series $W(\cdot; f,f)$ and $\zeta(\cdot)$ are considered only within their half-planes of absolute convergence.
Using the Gamma reflection property $\Gamma(z)\Gamma(1-z) = \pi \csc(\pi z)$ to raise the Gamma function in the denominator to the numerator and applying Stirling's approximation leads us to consider
\begin{equation}\label{eq:MB_integral_after_stirlings}
  X^{\frac{3}{2} - M_\delta} \int_{(\sigma_s)}\int_{(\sigma_z)} \lvert s - z \rvert^{k + \epsilon} \lvert s \rvert^{M_\delta - k} \mathcal{E}(s,z) \exp\left(\frac{\pi s^2}{y^2}\right) \frac{dz ds}{y},
\end{equation}
where $\mathcal{E}(s,z)$ are the collected exponential contributions,
\begin{equation*}
  \mathcal{E}(s,z) = \exp\!\left(-\tfrac{\pi}{2}\lvert \Im(s - z) \rvert - \tfrac{\pi}{2} \lvert \Im z \rvert + \tfrac{\pi}{2} \lvert \Im s \rvert\right).
\end{equation*}

We bound~\eqref{eq:MB_integral_after_stirlings} in cases, splitting the integral into intervals based on the signs and relative sizes of $\Im s$ and $\Im z$. The dominant contributions occur when $\mathcal{E}(s,z)$ provides no additional exponential decay, i.e.\ when $\Im s > \Im z > 0$ or $\Im s < \Im z < 0$. In this first case, let $t=\Im s$ and $u = \Im z$ and compute
\begin{equation}
  \begin{split}
    X^{\frac{3}{2}-M_\delta}\int_0^\infty \int_0^t (t-u)^{k+\epsilon} t^{M_\delta-k} &\exp\left(-\frac{\pi t^2}{y^2}\right) \frac{dudt}{y} \\
    &=O_\delta \left( X^{\frac{3}{2}} y^{1 + \epsilon} \left( \frac{y}{X} \right)^{M_\delta} \right).
  \end{split} \label{eq:MB_bound_with_M}
\end{equation}
The other cases are extremely similar and are also bounded by the error term in~\eqref{eq:MB_bound_with_M}.

Recalling that $y \ll X^{1 - \delta}$, we see that~\eqref{eq:MB_bound_with_M} can be made arbitrarily small for large $X$, although perhaps with a very large implicit constant depending on the size of $M_\delta$.
In particular, if we choose $M_\delta > \delta^{-1}(\tfrac{5}{2} + \epsilon)$, then~\eqref{eq:MB_bound_with_M} is non-dominant.

In conclusion, by combining the upper bounds~\eqref{eq:MB_trivial_first_residue_bound} and~\eqref{eq:MB_total_residue_bound} and noting that $X^{\frac{4}{3} + \epsilon} y^{-1} \gg X^{\frac{1}{3} + \epsilon}$ when $y \ll X^{1-\delta}$, we see that our initial integral~\eqref{eq:MB_integral_base_integral} is indeed
\begin{equation}
  O_{\delta, \epsilon} \left( \frac{X^{\frac{4}{3} + \epsilon} }{y} \right),
\end{equation}
as we set out to prove.

\section*{References}

\bibliographystyle{alpha}
\bibliography{shortintervalsbib}

\end{document}